\documentclass[small]{article}
\usepackage{amsmath,amstext,amsthm,amsfonts}
\usepackage{makeidx}
\usepackage{amssymb,latexsym}
\usepackage{txfonts, pxfonts}
\usepackage[english]{babel}
\usepackage[pdftex]{graphicx}
\usepackage{xcolor}

\numberwithin{equation}{section}
\newtheorem{theorem}{\textbf{Theorem}}[section]
\newtheorem{proposition}[theorem]{\textbf{Proposition}}

\newtheorem{lemma}[theorem]{\textbf{Lemma}}
\newtheorem{corollary}[theorem]{Corollary}

\theoremstyle{definition}

\theoremstyle{remark}
\newtheorem{remark}[theorem]{\it{Remark}}

\def\Rn{{\mathbb{R}}^n_+}
\def\d{\partial}

\def\ba{\begin{align}}

\def\ba{\begin{align}}
\def\ea{\end{align}}
\def\bp{\begin{proof}}
\def\ep{\end{proof}}

\newcommand{\R}{\mathbb{R}}
\newcommand{\N}{\mathbb{N}}

\renewcommand{\(}{\left(}
\renewcommand{\)}{\right)}

\begin{document}
	
\title{Energy bounds of sign-changing solutions to Yamabe equations on manifolds with boundary}
\date{}

\author{\textsc{S\'ergio Almaraz}\footnote{Partially supported by FAPERJ-202.802/2019.} \textsc{and Shaodong Wang}\footnote{Partially supported by NSFC-12001364, NSFC-12031012 and the Institute of Modern Analysis-A
Frontier Research Center of Shanghai.}}

\maketitle

\begin{abstract}
We study the Yamabe equation in the Euclidean half-space. We prove that any sign-changing solution has at least twice the energy of a standard bubble. Moreover, a sharper energy lower bound of the sign-changing solution set is also established via the method of moving planes. This bound increases the energy range for which Palais-Smale sequences of related variational problem has a non-trivial weak limit.
\end{abstract}
\section{Introduction}
In this paper, we study the following semi-linear equation with Neumann-type boundary condition in the $n$-dimensional Euclidean half-space $ \R_+^n:=\{(x_1,...,x_n)\in\mathbb R^n\:|\:x_n\geqslant 0\}$, $n\geqslant 3$:
\begin{align}\label{eq:r+n}
\begin{cases}
\Delta u=0 \ & \text{in} \ \R^n_+\\
\frac{\partial u}{\partial \nu}+|u|^{\frac{2}{n-2}}u=0\ & \text{on} \ \partial \R_+^n,
\end{cases}
\end{align}
where $\nu$ is the inward-pointing normal vector. We are interested in the set of sign-changing solutions to equation \eqref{eq:r+n}. 

Equation~\eqref{eq:r+n} is the boundary version of the classical Yamabe problem, in the scalar-flat case. Known as the generalization of the famous uniformization theorem for Riemann surfaces, the Yamabe problem concerns the existence of conformal metrics with constant scalar curvature on a given smooth compact Riemannian manifold $(M,g)$ without boundary. If one writes the conformal change in the form $\tilde{g}=u^{\frac{4}{n-2}}g$ for some smooth positive function $u$, the problem is equivalent to proving the existence of positive solutions to the following equation:
\begin{equation}\label{eq:yp}
\Delta_gu-\frac{n-2}{4(n-1)}R_gu+K|u|^{\frac{4}{n-2}}u=0  \ \text{in} \ M,
\end{equation}
where $\Delta_g:=div_g\nabla$ is the Laplace-Beltrami operator, $R_g$ denotes the scalar curvature of $(M,g)$ and $K$ here corresponds to the prescribed constant scalar curvature of the conformal metric $\tilde{g}$. Existence of positive solutions to ~\eqref{eq:yp} has been obtained by the combined work of Aubin~\cite{aub}, Schoen~\cite{sch}, Trudinger~\cite{tru} and Yamabe~\cite{yam}. The study of multiplicity of positive solutions has also drawn wide attention ever since it was brought up by Schoen in his Stanford lectures in 1988. He asked if the set of positive solutions to \eqref{eq:yp} is compact in the sense that it is bounded in the $C^{2,\alpha}$-topology for some $0<\alpha<1$. This problem has been studied extensively by many mathematicians and was finally solved by Khuri, Marques and Schoen in \cite{khu-mar-sch}, where they showed that the set is compact if the dimension of the manifold satisfies $3\leqslant n\leqslant 24$. On the other hand, lack of compactness was proved by Brendle~\cite{bre} and Brendle and Marques~\cite{bre-mar} when $n\geqslant 25$. 

Less is known about the set of sign-changing solutions to equation~\eqref{eq:yp}. In a classical paper by Ding~\cite{din}, he established the existence of infinitely many sign-changing solutions to \eqref{eq:yp} in the Euclidean space, or equivalently, on the standard sphere, with unbounded energy. Using a variation of the method of moving planes, Weth~\cite{wet} obtained a sharper energy lower bound of the set of sign-changing solutions. For further references on the existence and multiplicity of sign-changing solutions to the Yamabe equation on general Riemannian manifolds, see for example Ammann and Humbert~\cite{amm-hum}, Clapp and Fern\'{a}ndez~\cite{cla-fer}, Clapp, Salda\~na and Szulkin~\cite{cla-sal-szu},  del Pino, Musso, Pacard and Pistoia~\cite{del-mus-pac-pis1, del-mus-pac-pis2}, Fern\'{a}ndez and Petean~\cite{fer-pet}, Henry~\cite{hen}, Musso and Wei~\cite{mus-wei}, Petean~\cite{pet} and the references within.

In the case where $(M,g)$ has a non-empty boundary, Escobar~\cite{esc} proposed and studied the following boundary version of the Yamabe problem in the scalar-flat case. Given a smooth compact Riemannian manifold with boundary $\partial M$, is there a conformal metric with zero scalar curvature and constant boundary mean curvature? If we write similarly the conformal change as $\tilde{g}=u^{\frac{4}{n-2}}g$, then the problem is equivalent to proving the existence of positive solutions to the following equation:
\begin{align}\label{eq:ypb}
\begin{cases}
\Delta_g u-\frac{n-2}{4(n-1)}R_g u=0 \ & \text{in} \ M\\
\frac{\partial u}{\partial \nu_g}-\frac{n-2}{2}h_g u+K|u|^{\frac{2}{n-2}}u=0\ & \text{on} \ \partial M,
\end{cases}
\end{align}
where $\nu_g$ is the inward-pointing normal vector, $h_g$ denotes the boundary mean curvature, and $K$ now corresponds to the prescribed constant mean curvature with respect to $\tilde{g}$. Regularity of solutions to \eqref{eq:ypb} was obtained by Cherrier~\cite{che}. Existence results were established by Almaraz~\cite{alm}, Brendle and Chen~\cite{bre-che}, Escobar~\cite{esc}, Marques~\cite{mar1, mar2} and Mayer and Ndiaye~\cite{may-ndi}. For compactness and non-compactness results, see for example Almaraz~\cite{alm3, alm2}, Almaraz, Queiroz and Wang~\cite{alm-que-wan}, Felli and Ahmedou~\cite{fel-ahm1, fel-ahm2}, Ghimenti and Micheletti~\cite{ghi-mic} and Kim, Musso and Wei~\cite{kim-mus-wei}.

Equation \eqref{eq:r+n} is called critical due to the lack of compactness of the corresponding Sobolev trace embedding. As a result, the traditional variational method cannot be applied directly here to prove the existence of solutions. Solutions to \eqref{eq:r+n} are critical points of the functional
\begin{equation}\label{energy:r+n}
I(u)=\frac{1}{2}\int_{\R^n_+}|\nabla u|^2-\frac{1}{2^*}\int_{\partial\R^n_+}|u|^{2^*},
\end{equation}
where $2^*:=\frac{2(n-1)}{n-2}$ denotes the critical power of the following Sobolev trace inequality:
\begin{equation}\label{ineq:sob}
\int_{\R^n_+}|\nabla u|^2\geqslant S_n\(\int_{\partial\R^n_+}|u|^{2^*}\)^{\frac{2}{2^*}}.
\end{equation}
Here $S_n$ is the Sobolev best constant. 

Define $D^{1,2}(\R_+^n)$ to be the completion of smooth functions with compact support with respect to the norm $$\|u\|_D=\(\int_{\R_+^n}|\nabla u|^2\)^\frac{1}{2}.$$ Throughout the paper, we use the term  standard bubbles to denote the following functions:
\begin{equation}\label{eq:bubble}
U(\varepsilon,y;x,t):=\(\frac{(n-2)\varepsilon}{(\varepsilon+t)^2+|x-y|^2}\)^{\frac{n-2}{2}} \ \text{for} \ x\in  \R^{n-1}, \ t\geqslant  0,
\end{equation}
where $\varepsilon>0$, $y\in \R^{n-1}$. We also write 
$$B:=\{U(\varepsilon,y;x,t):\varepsilon>0, y\in \R^{n-1}\}$$ as the set of all standard bubbles. We sometimes use $U(\varepsilon,y)$ to denote $U(\varepsilon,y;x,t)$ when there is no possible confusion. It is well known that $B$ is the set of all the positive solutions to \eqref{eq:r+n}.  Moreover, $B$ is also the set of least energy critical points of $I$ with $I(U(\varepsilon,y))=\frac{1}{2(n-1)}S_n^{n-1}$ for any $\varepsilon>0$, $y\in \R^{n-1}$. Equality holds in \eqref{ineq:sob} if and only if $u$ takes the form of standard bubbles \eqref{eq:bubble} up to a constant multiple. See for example Escobar~\cite{esc1}.

Via the extension method developed by Caffarelli and Silvestre~\cite{caf-sil}, the existence of a solution $u$ to \eqref{eq:r+n} is equivalent to the existence of a solution $\bar{u}$ to the following fractional equation:
\begin{equation}\label{eq:frac}
(-\Delta)^{\frac{1}{2}}\bar{u}=|\bar{u}|^\frac{2}{n-2}\bar{u} \ \text{in} \ \R^{n-1},
\end{equation}
where $\bar{u}=u|_{\partial \R_+^n}$. Just recently, using a similar method as in Ding~\cite{din}, Abreu, Barbosa and Ramirez~\cite{abr-bar-ram} established the existence of sign-changing solutions to \eqref{eq:frac} with unbounded energy. In particular, they proved the following:

\begin{theorem}[\cite{abr-bar-ram}]\label{thm:existence frac}
Equation~\eqref{eq:frac} has an unbounded sequence of sign-changing solutions $\{\bar{u}_k\}_{k\in \N}\subset D^{\frac{1}{2},2}(\R^{n-1})$ when $n\geqslant 4$.
\end{theorem}

Here, the space $D^{s,2}(\R^n), \ 0<s<1$, is defined to be the completion of smooth functions with compact support with respect to the norm $$\|u\|_{s}:=\(\int_{\R^n}u(-\Delta)^su\)^\frac{1}{2}.$$ For more details on the fractional Laplacian operator and related function spaces, see Section 2. Using the extension method, it is easy to see that an unbounded sequence of solutions to \eqref{eq:frac} in $D^{\frac{1}{2},2}(\R^{n-1})$ corresponds to an unbounded sequence of solutions to \eqref{eq:r+n} in $D^{1,2}(\R_+^n)$.

As a corollary of Theorem~\ref{thm:existence frac}, we can obtain the following existence result of sign-changing solutions to equation~\eqref{eq:r+n}:

\begin{corollary}\label{thm:existence}
	There exists a sequence of sign-changing solutions $\{u_k\}_{{k\in \N}} \subset D^{1,2}(\R_+^n)$ of \eqref{eq:r+n} such that $\lim\limits_{k\to +\infty}I(u_k)\rightarrow +\infty$ when $n\geqslant 4$.
\end{corollary}

Corollary~\ref{thm:existence} establishes the existence of sign-changing solutions of different energy levels. As we have already mentioned, all positive solutions of \eqref{eq:r+n}, i.e., all standard bubbles, have the same level of energy $\frac{1}{2(n-1)}S_n^{n-1}$ and they are the set of least energy solutions. The following proposition states that any sign-changing solution to \eqref{eq:r+n} has at least twice the energy of the standard bubbles: 

\begin{proposition}\label{prop:lb}
	Every sign-changing solution $u \in D^{1,2}(\R_+^n)$ of \eqref{eq:r+n} satisfies $I(u)>\frac{1}{n-1}S_n^{n-1}$ where $S_n$ is the Sobolev best constant defined in \eqref{ineq:sob}.
\end{proposition}

A natural question is whether the energy lower bound in Proposition~\ref{prop:lb} is sharp. Our next result is inspired by \cite{wet} and gives a negative answer to that.

\begin{theorem}\label{thm:lb}
	There exists $\gamma>0$ such that $I(u)\geqslant \frac{1}{n-1}S_n^{n-1}+\gamma$ for any sign-changing solution $u \in D^{1,2}(\R_+^n)$ of \eqref{eq:r+n}.
\end{theorem}

As an application of Theorem~\ref{thm:lb}, let us consider the following equation in a bounded domain $D\subset \R^n$ with smooth boundary $\partial D$:
\begin{align}\label{eq:D}
\begin{cases}
\Delta u=0 \ & \text{in} \ D\\
\frac{\partial u}{\partial \nu}+\lambda u+|u|^{\frac{2}{n-2}}u=0\ & \text{on} \ \partial D,
\end{cases}
\end{align}
where $\lambda\in \R$ and $\nu$ points inwards. The corresponding functional is 
\begin{equation}\label{energy:D}
I_{\lambda,D}(u)=\frac{1}{2}\int_{D}|\nabla u|^2-\frac{\lambda }{2}\int_{\partial D}u^2-\frac{1}{2^*}\int_{\partial D}|u|^{2^*}.
\end{equation}
A slight modification of \cite[Theorem 1.3]{alm4} to include sign-changing functions produces the following Struwe-type compactness result:

\begin{theorem}[\cite{alm4}]\label{struwe_thm}
Suppose a sequence $\{u_k\}_{k\in\mathbb{N}}\subset H^1(D)$ is such that $\{I_{\lambda,D}(u_{k})\}$ is bounded and 
$$\nabla I_{\lambda,D}(u_k)\rightarrow 0\,\:\:\:\text{as}\:k\to+\infty.$$
Then there exist $m\in\{0,1,2,...\}$, a solution  $u^0\in H^1(D)$ of \eqref{eq:D}, $m$ non-trivial solutions $u^{(j)}\in D^{1,2}(\Rn)$ of \eqref{eq:r+n}, sequences  $\{R_{k}^{(j)}>0\}_{k\in\mathbb{N}}$ and sequences $\{x_{k}^{(j)}\}_{k\in\mathbb{N}}\subset \d D$, $1\leqslant j\leqslant m$ such that the whole satisfies the following conditions for $1\leqslant j\leqslant m$, possibly after taking a subsequence:
\\
(i) $R_{k}^{(j)}\to+\infty$ as $k\to+\infty$;
\\
(ii) $x_{k}^{(j)}$ converges as $k\to+\infty$;
\\
(iii) $\left\| u_{k}-u^0-\sum_{j=1}^{m}u_{k}^{(j)}\right\|_{H^1(D)}\to 0$ as $k\to+\infty$, where 
$$
u_{k}^{(j)}(x)=(R_{k}^{(j)})^{\frac{n-2}{2}}u^{(j)}(R_{k}^{(j)}(x-x_k^{(j)}))\,.
$$
Moreover,
$$
I_{\lambda,D}(u_{k})-I_{\lambda,D}(u^0)-\sum_{j=1}^{m}I(u_{k}^{(j)})\to 0
\:\:\:\:\text{as}\:k\to+\infty\,. $$
\end{theorem}
Hence, Theorem \ref{struwe_thm}, together with Theorem~\ref{thm:lb}, gives the following:
\begin{corollary}\label{cor PS}
	Let $\Lambda:=min\{\gamma, \frac{1}{2(n-1)}S_n^{\frac{1}{n-1}}\}$, where $\gamma$ is obtained in Theorem~\ref{thm:lb} and $S_n$ is the Sobolev best constant defined as before. If $\{u_k\}_{k\in \N}\subset H^1(D)$ is a sequence with
	$$\nabla I_{\lambda,D}(u_k)\rightarrow 0,$$
	and $$I_{\lambda,D}(u_k)\rightarrow c\in \(0,\frac{1}{n-1}S_n^{n-1}+\Lambda\)\backslash\left\lbrace \frac{1}{2(n-1)}S_n^{n-1},\frac{1}{n-1}S_n^{n-1}\right\rbrace ,$$
as $k\to+\infty$, 
then  a subsequence of $\{u_k\}_{k\in \N}$ has a non-trivial weak limit.
\end{corollary}

\begin{remark}
We would like to add that Corollary~\ref{cor PS} can be generalized to equations similar to \eqref{eq:ypb} on compact Riemannian manifolds with boundary. 
\end{remark}

For the proof of Theorem~\ref{thm:lb}, we use a certain variation of the moving planes method. This method was originated from classical papers of Serrin~\cite{ser} and Gidas, Ni and Nirenberg~\cite{gid-ni-nir}. It was usually applied to obtain certain symmetry properties of solutions. In our paper, we use this method to rule out the possibility of sign-changing solutions consisting of two bubbles of opposite signs and consequently establish a sharper lower bound of the energy. 

The paper is organized as follows. In Section 2, we present some basics about equation~\eqref{eq:r+n} and the related fractional operator and function spaces. Some elementary lemmas involving the convergence and transformation of certain functions will also be obtained in this section. In Section 3, we first prove Proposition~\ref{prop:lb} using a direct variational argument. The proof of Theorem~\ref{thm:lb} is given in the rest of that section via the method of moving planes.

\section{Preliminaries}
Throughout the paper, we denote $u^+:=max\{u,0\}$ and  $u^-:=max\{-u,0\}$ for any $u\in D^{1,2}(\R_+^n)$ so that $u=u^+-u^-$. Define $C_B^1(\R_+^n)$ to be the space of all bounded functions in $C^1(\R_+^n)$ with bounded gradient, endowed with the norm $\|u\|_{C_B^1(\R_+^n)}=\|u\|_{L^{\infty}(\R^n_+)}+\|\nabla u\|_{L^{\infty}(\R^n_+)}$.

We also define $\Gamma$ to be the set of all finite compositions of the following transformations on $u\in D^{1,2}(\R_+^n)$ : the translations, rotations, rescalings and Kelvin transformation, respectively
$$(y*u)(x,t):=u(x-y,t), \ y\in \R^{n-1},$$
$$(A*u)(x,t):=u(A^{-1}x,t), \ A\in O(n-1),$$
$$(\varepsilon*u)(x,t):=\varepsilon^{-\frac{n-2}{2}}u\(\frac{x}{\varepsilon},\frac{t}{\varepsilon}\), \ \varepsilon>0,$$
$$(K*u)(x,t):=|(x,t)|^{2-n}u\(\frac{x}{|(x,t)|^2},\frac{t}{|(x,t)|^2}\),$$
for any $x\in \R^{n-1}$ and $t\geqslant 0$. Here $O(n-1)$ is the orthogonal group.

It is easy to see that the functional \eqref{energy:r+n} is invariant under any transformation of $\Gamma$ and any $T\in \Gamma$ maps solutions of \eqref{eq:r+n} to solutions of \eqref{eq:r+n}. The following identities are easily verified:
$$(y'*U)(\varepsilon,y;x,t)=U(\varepsilon,y+y';x,t), \ y'\in \R^{n-1},$$
$$(A*U)(\varepsilon,y;x,t)=U(\varepsilon,Ay;x,t), \ A\in O(n-1),$$
$$(\varepsilon'*U)(\varepsilon,y;x,t)=U(\varepsilon'\varepsilon,\varepsilon'y;x,t), \ \varepsilon'>0,$$
$$(K*U)(\varepsilon,y;x,t)=U\(\frac{\varepsilon}{{\varepsilon}^2+|y|^2},\frac{y}{{\varepsilon}^2+|y|^2};x,t\),$$
where $U(\varepsilon,y;x,t)$ are the standard bubbles defined in \eqref{eq:bubble}.

The fractional Laplacian in $\mathbb{R}^n$ is a non-local pseudo-differential operator of the form:
\begin{equation}\label{eq:fl}
\begin{aligned}
(-\Delta)^s u(x)&=C_{n,s}P.V.\int_{\mathbb{R}^n}\frac{u(x)-u(y)}{|x-y|^{n+2s}}dy\\
&=C_{n,s}\lim_{\epsilon\rightarrow0^+}\int_{\mathbb{R}^n\backslash B_{\epsilon}(x)}\frac{u(x)-u(y)}{|x-y|^{n+2s}}dy,
\end{aligned}
\end{equation}
where $0<s<1$, $C_{n,s}$ is a positive constant depending only on $s$ and $n$ and $P.V.$ stands for the Cauchy principle value. In order for the integral to make sense, we require $u\in L^{2s}\cap C^{1,1}_{\rm loc}(\mathbb{R}^n)$ where
\begin{equation*}
L^{2s}(\mathbb{R}^n)=\{u\in L_{loc}^1(\mathbb{R}^n) \:| \int_{\mathbb{R}^n}\frac{|u(x)|}{1+|x|^{n+2s}}<+\infty\}.
\end{equation*}

Let's start by proving some elementary lemmas that we need in the next section. 

\begin{lemma}\label{lemma:regularity}
	Suppose $\{u_k\}_{k\in \N}\subset D^{1,2}(\R_+^n)$ is a sequence of solutions to \eqref{eq:r+n} with $sup_{k\in \N}\|u_k\|_D<+\infty$ and 
	\begin{equation}\label{eq:2^*bound}
	\lim_{k\rightarrow +\infty}\int_{B_{r_k}(x_k)\cap \partial \R_+^n}|u_k|^{2^*}=0,
	\end{equation}
	for any $x_k\in \partial \R_+^n$, and $r_k>0$ with $r_k\rightarrow 0$. Then we have $sup_{k\in \N}\|u_k\|_{C_B^1(\R_+^n)}<+\infty$.
\end{lemma}

\begin{proof}
	In view of $L^p$-estimates and Sobolev embedding theories, it is enough to prove that $\{u_k\}_{k\in \N}\subset D^{1,2}( \R_+^n)$ is bounded in $L^{\infty}(\R_+^n)$.
	Let us assume by contradiction that there exists a sequence of points $\{x_k\}\subset \partial \R_+^n$ such that $$s_k:=\|u_k\|_{L^{\infty}(\R^n_+)}=|u_k(x_k)|\rightarrow +\infty.$$ Let $v_k(x):=s_k^{-1}u_k\(s_k^{-\frac{2}{n-2}}x+x_k\).$ By the scaling invariance of equation \eqref{eq:r+n} we see that $\{v_k\}_{k\in \N}$ is a sequence of solutions to \eqref{eq:r+n} with $\|v_k\|_{L^{\infty}(\R^n_+)}=|v_k(0)|=1$ for all $k$. Therefore again in view of $L^p$-estimates and Sovolev embedding theories we know that $\{v_k\}_{k\in \N}$ is bounded in $C_B^1(\R_+^n)$. Let's assume that up to a subsequence, still denoted as $v_k$, $v_k\rightarrow v$ in $C(B_1(0)\cap \R_+^n)$. Let $r_k:=s_k^{-\frac{2}{n-2}}\rightarrow 0$ as $k\rightarrow +\infty$, we obtain using $\|v_k\|_{L^{\infty}(\R^n_+)}=|v_k(0)|=1$:
	$$\int_{B_{r_k}(x_k)\cap \partial \R_+^n}|u_k|^{2^*}=\int_{B_{1}(0)\cap \partial \R_+^n}|v_k|^{2^*}\rightarrow\int_{B_{1}(0)\cap \partial \R_+^n}|v|^{2^*}>0,$$ which is a contradiction to \eqref{eq:2^*bound}. This completes the proof.
\end{proof}

Now let us use Lemma~\ref{lemma:regularity} to obtain the following convergence result:

\begin{lemma}\label{lemma:inftybound}
	Let $\{u_k\}_{k\in \N}\subset D^{1,2}( \R_+^n)$ be a sequence of solutions to \eqref{eq:r+n} such that $\|u_k-w_k\|_D\rightarrow 0$ for some sequence $\{w_k\}_{k\in \N}\subset D^{1,2}(\R_+^n)\cap C_B^1(\R_+^n)$ which is bounded in $C_B^1(\R_+^n)$. Then $\|u_k-w_k\|_{L^{\infty}(\R^n_+)}\rightarrow 0$ as $k\rightarrow +\infty$.
\end{lemma}

\begin{proof}
	Observe that for any sequence $\{x_k\}\subset \partial \R_+^n$, $r_k>0$, $r_k\rightarrow 0$ we have by Sobolev inequality \eqref{ineq:sob}:
	\begin{align*}
	\(\int_{B_{r_k}(x_k)\cap \partial \R_+^n}|u_k|^{2^*}\)^{\frac{1}{2^*}}&\leqslant \|u_k-w_k\|_{L^{2^*}(\partial \R_+^n)}+\(\int_{B_{r_k}(x_k)\cap \partial \R_+^n}|w_k|^{2^*}\)^{\frac{1}{2^*}}
	\\&\leqslant \|u_k-w_k\|_{L^{2^*}(\partial \R_+^n)}+o(1)\|w_k\|_{L^{\infty}(\R^n_+)}
	\\&\rightarrow 0.
	\end{align*}
	Thus it follows from Lemma~\ref{lemma:regularity} that $\{u_k\}_{k\in \N}$ is bounded in $C_B^1(\R_+^n)$. Assume by contradiction that for a subsequence, there exist points $\{z_k\}\subset \R_+^n$ such that
	\begin{equation}\label{eq:positive inf}
	\liminf_{k\rightarrow+\infty}|u_k(z_k)-w_k(z_k)|>0.
	\end{equation}
	Define $\bar{u}_k(x)=u_k(x+z_k)$ and $\bar{w}_k(x)=w_k(x+z_k)$, then it is easy to see that $\{\bar{u}_k\}_{k\in \N}$ and $\{\bar{w}_k\}_{k\in \N}$ are uniformly bounded in $C_B^1(\R_+^n)$. Moreover, $\|\bar{u}_k-\bar{w}_k\|_{D}\rightarrow 0$ in $\R_+^n$. 
	We may assume without loss of generality that $\bar{u}_k\rightarrow u$ in $C_{loc}(\R_+^n)$ and  $\bar{w}_k\rightarrow w$ in $C_{loc}(\R_+^n)$. Thus we have by \eqref{eq:positive inf}:
	$$|u(0)-w(0)|>0.$$ Therefore, we have that 
	$$\lim_{k\rightarrow+\infty}\int_{B_1(0)\cap \R_+^n} |\bar{u}_k(x)-\bar{w}_k(x)|=\int_{B_1(0)\cap \R_+^n} |u(x)-w(x)|>0.$$
	This contradicts the fact that $\|\bar{u}_k-\bar{w}_k\|_D\rightarrow 0$ in $\R_+^n$, hence $\bar{u}_k-\bar{w}_k\rightarrow 0$ in $L^1_{loc}(\R_+^n)$, therefore completing the proof of the lemma.
\end{proof}

The next lemma establishes certain transformation properties of the standard bubbles $U(\varepsilon,y)$. 
\begin{lemma}\label{lemma:trans}
	Let $\varepsilon, \varepsilon'>0$ and $y, y'\in \R^{n-1}$. Then there exists a transformation $T\in \Gamma$ such that 
	$$TU(\varepsilon, y)=U(1,z),$$
	$$TU(\varepsilon', y')=U(1,-z),$$
	where $z=(z_1,0,\cdot\cdot\cdot,0)\in \R^{n-1}$ for some $z_1>0$.
\end{lemma}

\begin{proof}
	We divide the proof into two cases:
	\begin{enumerate}
		\item[Case 1]: $\varepsilon=\varepsilon'$. After rescaling we may assume that $\varepsilon=\varepsilon'=1$. Then we can find suitable translation and rotation that transform $y$ and $y'$ into $z=(z_1,0,\cdot\cdot\cdot,0)$ and $-z=(-z_1,0,\cdot\cdot\cdot,0)$ for some $z_1>0$. Thus there exists a transformation $T$ in $\Gamma$ as required.
		\item[Case 2]: $\varepsilon\neq \varepsilon'$. Using proper rescaling, translation and rotation applied to both $U(\varepsilon, y)$ and $U(\varepsilon', y')$, we may assume that $(\varepsilon,y)=(1,0)$, $\varepsilon'>0$ and $y'=(y'_1,0,\cdot\cdot\cdot,0)$ for some $y'_1\in \R$. Apply a transformation of the form $K\circ \bar{y}$ where $K$ is the Kelvin transformation and $\bar{y}=(\bar{y}_1,0,\cdot\cdot\cdot,0)$ for some $\bar{y}_1\in \R$, and use the invariance properties of transformations in Section 2 to obtain
		$$(K\circ \bar{y})*U(1,0)=U\(\frac{1}{1+\bar{y}_1^2},\frac{\bar{y}}{1+\bar{y}_1^2}\),$$
		$$(K\circ \bar{y})*U(\varepsilon',y')=U\(\frac{\varepsilon'}{(\varepsilon')^2+(\bar{y}_1+y_1')^2},\frac{y'+\bar{y}}{(\varepsilon')^2+(\bar{y}_1+y_1')^2}\).$$
		It is elementary to see that 
		$$\frac{1}{1+\bar{y}_1^2}=\frac{\varepsilon'}{(\varepsilon')^2+(\bar{y}_1+y_1')^2}$$ 
		has a real solution $\bar{y}_1\in \R$. Let $\bar{y}_1$ be such a solution and set $\frac{1}{1+\bar{y}_1^2}=\varepsilon_1>0$. Then from above we know that there exists a transformation $T\in \Gamma$ such that $$TU(\varepsilon,y)=U(\varepsilon_1,\theta),$$
		$$TU(\varepsilon',y')=U(\varepsilon_1,\theta'),$$ for some $\theta, \theta'\in \R^{n-1}$.
		Now we can proceed as in Case 1 to complete our proof.
	\end{enumerate}
\end{proof}

\section{Energy lower bound}

We first prove Proposition~\ref{prop:lb} which states that any sign-changing solution to \eqref{eq:r+n} has at least twice the energy of the standard bubbles.

\begin{proof}[Proof of Proposition~\ref{prop:lb}]
	Let $u$ be a sign-changing solution to \eqref{eq:r+n}. Testing $u^{\pm}$ to the equation and using Sobolev trace inequality \eqref{ineq:sob}, we get
	\begin{align*}
	0&=\int_{\R_+^n}(-\Delta u)u^{\pm}=\int_{\R_+^n}|\nabla u^{\pm}|^2-\int_{\partial \R_+^n}|u^{\pm}|^{2^*}\\ &\geqslant (1-S_n^{-\frac{2^*}{2}}||u^{\pm}||_{D}^{2^*-2})||u^{\pm}||^2_{D},
	\end{align*}
	where $2^*=\frac{2(n-1)}{n-2}$ is the critical exponent of the Sobolev trace embedding defined as before. In particular, $\|u^{\pm}\|_{L^{2^*}(\partial \R_+^n)}^{2^*}=\|u^{\pm}\|^2_{D}\neq 0$ and $\|u^{\pm}\|^2_{D}\geqslant S_n^{n-1}$, we have
	\begin{equation*}
	I(u^{\pm})=\frac{1}{2(n-1)}||u^{\pm}||_{D}^2\geqslant \frac{1}{2(n-1)}S_n^{n-1}.
	\end{equation*}
	We claim that it must hold:
	\begin{equation*}
	I(u^{\pm})> \frac{1}{2(n-1)}S_n^{n-1}.
	\end{equation*}
	If this is not the case, let us assume that
	\begin{equation*}
	I(u^+)=\frac{1}{2(n-1)}S_n^{n-1},
	\end{equation*}
	then $u^+$ must be a solution to equation \eqref{eq:r+n}. It follows from maximum principle that $u^+>0$ which contradicts the fact that $u$ is sign-changing. Therefore we must have 
	\begin{equation*}
	I(u)=I(u^+)+I(u^-)>\frac{1}{n-1}S_n^{n-1}.
	\end{equation*}
	This ends the proof of Proposition~\ref{prop:lb}.
\end{proof}

In the rest of this section, we will prove the sharper energy bound of Theorem~\ref{thm:lb}. Let us first introduce some notations. Denote 
$$
H_s:=\{x\in \R_+^n: x_1>s, \ s\in \R\},
$$ 
$\partial'H_s:=\{x\in \partial \R_+^n: x_1>s\}$ and $\partial^+H_s:=\{x\in \partial H_s: x_1=s\}$. Denote also $r_s(x)$ to be the reflection of a point $x\in \R^n_+$ about the hyperplane $\{x_1=s\}$. Set 
\begin{equation}\label{eq:L}
L:=2^{-\frac{n}{n-2}}\Big(\frac{n-2}{n}S_n\Big)^{n-1},
\end{equation}
where $S_n$ is the Sobolev best constant as before.

The proof of Theorem~\ref{thm:lb} is based on the following two propositions. First, we use a variation of the moving planes method to prove the Proposition~\ref{prop:nonexistence} which establishes the non-existence of sign-changing solutions to \eqref{eq:r+n} of a special shape.

\begin{proposition}\label{prop:nonexistence}
	Let $H:=H_0$ be as defined before and $\Omega$ be a non-empty compact subset of $\partial'H$. Then there is no sign-changing solution $u\in D^{1,2}(\R_+^n)$ to \eqref{eq:r+n} such that the following three conditions hold:
	\begin{enumerate}
		\item $u(x)>u(r_0(x))$ for $x\in \Omega$,
		\item $\int_{\partial \R_+^n \setminus \{\Omega\cup r_0(\Omega)\}} |u|^{2^*}<L,$
		\item $\inf_{\partial'H}u>\inf_{\partial \R_+^n\setminus\partial'H}u.$
	\end{enumerate}
\end{proposition}

\begin{proof}
	Let us prove by contradiction. Assume that there exists a sign-changing solution $u$ to equation \eqref{eq:r+n} that satisfies conditions 1-3. Consider $w_s:=u-u\circ r_s$ where $r_s$ is the reflection about the hyperplane $\{x_1=s\}$. A direct calculation gives that $w_s$ satisfies:
	\begin{align}\label{eq:w}
	\begin{cases}
	\Delta w_s=0 \ & \text{in} \ \R^n_+\\
	\frac{\partial w_s}{\partial \nu}+V_sw_s=0\ & \text{on} \ \partial \R_+^n,
	\end{cases}
	\end{align}
	where $0\leqslant V_s \leqslant (2^*-1)(|u|+|u\circ r_s|)^{2^*-2}.$

Writing $w_0^-=max\{-w_0,0\}$, we have in particular that
\begin{align*}
	\begin{cases}
	\Delta w_0^-=0 \ & \text{in} \ \R^n_+\\
	\frac{\partial w_0^-}{\partial \nu}+V_0w_0^-=0\ & \text{on} \ \partial \R_+^n.
	\end{cases}
	\end{align*}
Since $w_0^-=0$ on $\partial^{+} H$ and in a neighbourhood of $\Omega$, an integration by parts gives
\begin{align*}
	|| w_0^-||^2_{D}=\int_{\partial 'H\setminus \Omega}V_0(w_0^-)^2
\leqslant ||V_0||_{L^{n-1}(\partial 'H\setminus \Omega)}||w_0^-||_{L^{2^*}(\partial 'H\setminus \Omega)}^2.
	\end{align*}
	On the other hand, by condition 2 we obtain
	\begin{align*}
	\int_{\partial'H \setminus \Omega}|V_0|^{n-1}&\leqslant (2^*-1)^{n-1}\int_{\partial'H \setminus \Omega}(|u|+|u\circ r_0|)^{2^*}\\& \leqslant (2^*-1)^{n-1}2^{2^*-1}\int_{\partial'\R_+^n \setminus \{\Omega\cup (r_0(\Omega))\}}|u|^{2^*}\\&< (2^*-1)^{n-1}2^{2^*-1} L.
	\end{align*}
So, by the Sobolev trace inequality \eqref{ineq:sob},
\begin{equation}\label{ineq:L}
\begin{aligned}
	S_n||w_0^-||^2_{L^{2^*}(\partial 'H\setminus \Omega)}&\leqslant ||V_0||_{L^{n-1}(\partial 'H\setminus \Omega)}||w_0^-||_{L^{2^*}(\partial 'H\setminus \Omega)}^2
\\&<((2^*-1)^{n-1}2^{2^*-1}L)^{\frac{1}{n-1}}||w_0^-||_{L^{2^*}(\partial 'H\setminus \Omega)}^2.
\end{aligned}
\end{equation}
	Since $((2^*-1)^{n-1}2^{2^*-1}L)^{\frac{1}{n-1}}=S_n$ by the definition of $L$, we get
	$w_0^-\equiv 0$ on $\partial' H\setminus \Omega$. In particular, $w_0\geqslant 0$ in $\partial' H$. Set 
	$$\bar{s}=\sup\{s\in \R: w_s\geqslant 0 \ \text{in}\  \partial' H_s\}\geqslant 0.$$
	Since $u>0$ somewhere on $\partial \R_+^n$ and $u\rightarrow 0$ as $|x|\rightarrow +\infty$, $w_s<0$ for some $x$ in $\partial'H_s$ for large $s$. Thus $\bar{s}<+\infty$. By continuity $w_{\bar{s}}\geqslant 0$ in $\partial' H_{\bar{s}}$. Since $w_{\bar{s}}=0$ on $\partial^+H_{\bar{s}}$, it follows from the strong maximum principle that either $w_{\bar{s}}\equiv0$ or $w_{\bar{s}}>0$ in $\partial' H_{\bar{s}}$.  But $w_{\bar{s}}\equiv0$ is impossible because condition 3 implies that
	$$\inf_{\partial'H_{\bar{s}}}u\geqslant\inf_{\partial'H}u>\inf_{\partial \R_+^n\setminus\partial'H}u\geqslant\inf_{\partial\R^n_+\setminus\partial'H_{\bar{s}}}u.$$
	Thus we have $w_{\bar{s}}>0$ on $\partial' H_{\bar{s}}$. Choose a sufficiently large compact set $\Omega_1\subset\partial' H_{\bar{s}}$ such that 
	$$\int_{\partial \R_+^n \setminus \{\Omega_1\cup r_{\bar{s}}(\Omega_1)\}} |u|^{2^*}<L.$$
We can also choose  $s_1>\bar{s}$ close to $\bar{s}$ such that $\Omega_1\subset\partial' H_{s_1}$, $w_{s_1}>0$ in $\Omega_1$ and 
	$$\int_{\partial \R_+^n \setminus \{\Omega_1\cup r_{s_1}(\Omega_1)\}} |u|^{2^*}<L.$$
Using an inequality similar to \eqref{ineq:L} and arguing as before we conclude that $w_{s_1}\geqslant 0$ on $\partial'H_{s_1}$. This contradicts the definition of $\bar{s}$. Thus we have finished the proof.
\end{proof}

The second part of the proof of Theorem~\ref{thm:lb} is the following approximation result:

\begin{proposition}\label{prop:approximation}
	Let $\{u_k\}_{k\in \N}$ be a sequence of sign-changing solutions of \eqref{eq:r+n} such that $I(u_k)\rightarrow \frac{1}{n-1}S^{n-1}$. Then there exists a sequence of transformations $\{T_k\}_{k\in \N}$ in $\Gamma$ and positive numbers $\{z_{k,1}\}_{k\in \N}$ such that, up to a subsequence, $z_{k,1}\rightarrow +\infty$, $$||T_ku_k-U(1,z_k)+U(1,-z_k)||_D\rightarrow 0,$$ and 
	$$||T_ku_k-U(1,z_k)+U(1,-z_k)||_{L^{\infty}(\R_+^n)}\rightarrow 0,$$ where $z_k=(z_{k,1},0,\cdot\cdot\cdot, 0)\in \R^{n-1}$.
\end{proposition}

\begin{proof}
	For all $k$ we test $u_k^{\pm}$ to equation \eqref{eq:r+n} to obtain
	\begin{equation}\label{eq:pm}
	I(u_k^{\pm})=\frac{1}{2(n-1)}||u_k^{\pm}||_D^2=\frac{1}{2(n-1)}||u_k^{\pm}||_{L^{2^*}(\partial \R_+^n)}^{2^*}\geqslant \frac{1}{2(n-1)}S_n^{n-1}.
	\end{equation}
	By the assumption of the proposition we obtain
	\begin{equation*}
	I(u_k)=I(u_k^+)+I(u_k^-)\rightarrow \frac{1}{n-1}S_n^{n-1}.
	\end{equation*}
	So we must have 
	$$I(u_k^{\pm})\rightarrow \frac{1}{2(n-1)}S_n^{n-1}.$$
	Again by equation \eqref{eq:pm} we have
	$$\frac{||u_k^{\pm}||_D^2}{||u_k^{\pm}||_{L^{2^*}(\partial \R_+^n)}^{2}}
=\big(2(n-1)I(u_k^{\pm})\big)^{\frac{1}{n-1}}\rightarrow S_n.$$
	Thus it follows from classical results by Escobar~\cite{esc1} and Lions~\cite{lio} that there exist $\varepsilon_k^{{(1)}}, \varepsilon_k^{{(2)}}>0$, $y_k^{(1)}, y_k^{(2)}\in \R^{n-1}$ such that 
	$$||u_k^+-U(\varepsilon_k^{(1)},y_k^{(1)})||_D\rightarrow 0,$$
	$$||u_k^--U(\varepsilon_k^{(2)},y_k^{(2)})||_D\rightarrow 0.$$
	As a result,
	$$||u_k-U(\varepsilon_k^{(1)},y_k^{(1)})+U(\varepsilon_k^{(2)},y_k^{(2)})||_D\rightarrow 0.$$
	From Lemma~\ref{lemma:trans} there exists a transformation $T_k\in \Gamma$ such that $T_kU(\varepsilon_k^{(1)},y_k^{(1)})=U(1,z_k)$ and $T_kU(\varepsilon_k^{(2)},y_k^{(2)})=U(1,-z_k)$ for some $z_k=(z_{k,1},0,\cdot\cdot\cdot,0)$ where $z_{k,1}>0$. Therefore
	$$||T_ku_k-U(1,z_k)+U(1,-z_k)||_D\rightarrow 0.$$ Using Lemma~\ref{lemma:inftybound} we have 
	$$||T_ku_k-U(1,z_k)+U(1,-z_k)||_{L^{\infty}(\R_+^n)}\rightarrow 0.$$
	Now let's prove that $z_{k,1}\rightarrow +\infty$. Suppose not, then up to a subsequence we assume $z_{k,1}\rightarrow z_1\geqslant 0$. If $z_1=0$, then $T_ku_k\rightarrow0$ in $D^{1,2}(\R_+^n)$ which is impossible since $I(T_ku_k)\rightarrow \frac{1}{n-1}S_n^{n-1}$. If $z_1>0$, then $u=U(1,z_1)-U(1,-z_1)$ is a sign-changing solution to equation \eqref{eq:r+n} with $I(u)= \frac{1}{n-1}S_n^{n-1}$. This contradicts Proposition \ref{prop:lb}. Therefore we must have  $z_{k,1}\rightarrow +\infty$.
\end{proof}

Now we are ready to prove Theorem~\ref{thm:lb}.
\begin{proof}[Proof of Theorem~\ref{thm:lb}]
	Assume by contradiction that there exists a sequence of solutions $\{u_k\}_{k\in \N}\subset D^{1,2}(\R_+^n)$ such that $I(u_k)\rightarrow \frac{1}{n-1}S_n^{n-1}$. Then by Proposition~\ref{prop:approximation}, there exists a sequence of transformations $T_k\in \Gamma$ such that
	$$||T_ku_k-U(1,z_k)+U(1,-z_k)||_D\rightarrow 0,$$ and 
	\begin{equation}\label{eq:infty approximation}
	||T_ku_k-U(1,z_k)+U(1,-z_k)||_{L^{\infty}(\R_+^n)}\rightarrow 0,
	\end{equation}
	where $z_k=(z_{k,1},0,\cdot\cdot\cdot, 0)\in \R^{n-1}$, $z_{k,1}\rightarrow +\infty$. Let $v_k=T_ku_k$. Then using the invariance properties of equation \eqref{eq:r+n} we know $\{v_k\}_{k\in \N}$ are sign-changing solutions to \eqref{eq:r+n} as well. Choose $R>0$ large enough such that
	$$\(\int_{\partial \R_+^n \setminus B_R(0)}U(1,0)^{2^*}\)^{\frac{1}{2^*}}<\frac{L^{\frac{1}{2^*}}}{2},$$
	where $L$ is defined in \eqref{eq:L}. Since $||v_k-U(1,z_k)+U(1,-z_k)||_{L^{2^*}(\partial \R_+^n)}\rightarrow 0$ and $z_{k,1}\rightarrow +\infty$, by Sobolev inequality \eqref{ineq:sob}, we know
	 $$\int_{\partial \R_+^n \setminus \{ B_R(z_k)\cup r_0(B_R(z_k)\}} |v_k|^{2^*}<L,$$ 
	 for $k$ large enough. Moreover, by \eqref{eq:infty approximation} we have $v_k>0$ in $B_R(z_k)\cap \partial \R_+^n$, $v_k<0$ in $r_0(B_R(z_k))\cap \partial \R_+^n$ and $\inf_{\partial'H}v_k>\inf_{\partial \R_+^n\setminus\partial'H}v_k$ for $k$ large enough.
	However it follows from Proposition~\ref{prop:nonexistence} that such solutions $\{v_k\}_{k\in \N}$ of \eqref{eq:r+n} do not exist. This gives a contradiction and finishes the proof of Theorem~\ref{thm:lb}.

\end{proof}

\bigskip\noindent
\textsc{S\'ergio Almaraz\\
	Instituto de Matem\'atica e Estat\' istica, 
	Universidade Federal Fluminense\\
	Rua Prof. Marcos Waldemar de Freitas S/N,
	Niter\'oi, RJ,  24.210-201, Brazil.}\\
e-mail: {\bf{sergio.m.almaraz@gmail.com}}

\bigskip\noindent
\textsc{Shaodong Wang\\
	School of Mathematical Sciences, Shanghai Jiao Tong University \\
	800 Dongchuan Road,
	Shanghai, China.}\\
e-mail: {\bf{shaodong.wang@sjtu.edu.cn}}


\begin{thebibliography}{1}
\bibitem{abr-bar-ram}
Emerson Abreu, Ezequiel Barbosa, and Joel Cruz Ramirez.
\newblock Infinitely many sign-changing solutions of a critical fractional equation.
\newblock {\em Ann. Mat. Pura Appl.}, 2021. https://doi.org/10.1007/s10231-021-01141-2.	

\bibitem{alm}
S\'{e}rgio de~Moura Almaraz.
\newblock An existence theorem of conformal scalar-flat metrics on manifolds
with boundary.
\newblock {\em Pacific J. Math.}, 248(1):1--22, 2010.

\bibitem{alm3}
S\'{e}rgio de~Moura~Almaraz.
\newblock Blow-up phenomena for scalar-flat metrics on manifolds with boundary.
\newblock {\em J. Differential Equations}, 251(7):1813--1840, 2011.

\bibitem{alm2}
S\'{e}rgio de~Moura~Almaraz.
\newblock A compactness theorem for scalar-flat metrics on manifolds with
boundary.
\newblock {\em Calc. Var. Partial Differential Equations}, 41(3-4):341--386,
2011.

\bibitem{alm4}
S\'{e}rgio Almaraz.
\newblock The asymptotic behavior of {P}alais-{S}male sequences on manifolds
with boundary.
\newblock {\em Pacific J. Math.}, 269(1):1--17, 2014.
	
\bibitem{alm-que-wan}
S\'{e}rgio Almaraz, Olivaine~S. de~Queiroz, and Shaodong Wang.
\newblock A compactness theorem for scalar-flat metrics on 3-manifolds with
boundary.
\newblock {\em J. Funct. Anal.}, 277(7):2092--2116, 2019.

\bibitem{amm-hum}
B.~Ammann and E.~Humbert.
\newblock The second {Y}amabe invariant.
\newblock {\em J. Funct. Anal.}, 235(2):377--412, 2006.

\bibitem{aub}
Thierry Aubin.
\newblock \'{E}quations diff\'{e}rentielles non lin\'{e}aires et probl\`eme de
{Y}amabe concernant la courbure scalaire.
\newblock {\em J. Math. Pures Appl. (9)}, 55(3):269--296, 1976.


\bibitem{bre}
Simon Brendle.
\newblock Blow-up phenomena for the {Y}amabe equation.
\newblock {\em J. Amer. Math. Soc.}, 21(4):951--979, 2008.

\bibitem{bre-che}
Simon Brendle and Szu-Yu~Sophie Chen.
\newblock An existence theorem for the {Y}amabe problem on manifolds with
boundary.
\newblock {\em J. Eur. Math. Soc. (JEMS)}, 16(5):991--1016, 2014.

\bibitem{bre-mar}
Simon Brendle and Fernando~C. Marques.
\newblock Blow-up phenomena for the {Y}amabe equation. {II}.
\newblock {\em J. Differential Geom.}, 81(2):225--250, 2009.

\bibitem{caf-sil}
Luis Caffarelli and Luis Silvestre.
\newblock An extension problem related to the fractional {L}aplacian.
\newblock {\em Comm. Partial Differential Equations}, 32(7-9):1245--1260, 2007.

\bibitem{che}
Pascal Cherrier.
\newblock Probl\`emes de {N}eumann non lin\'{e}aires sur les vari\'{e}t\'{e}s
riemanniennes.
\newblock {\em J. Funct. Anal.}, 57(2):154--206, 1984.

\bibitem{cla-fer}
M\'{o}nica Clapp and Juan~Carlos Fern\'{a}ndez.
\newblock Multiplicity of nodal solutions to the {Y}amabe problem.
\newblock {\em Calc. Var. Partial Differential Equations}, 56(5):Paper No. 145,
22, 2017.

\bibitem{cla-sal-szu}
M\'{o}nica Clapp, Alberto Salda\~{n}a, and Andrzej Szulkin.
\newblock Phase separation, optimal partitions, and nodal solutions to the
{Y}amabe equation on the sphere.
\newblock {\em Int. Math. Res. Not. IMRN}, (5):3633--3652, 2021.

\bibitem{del-mus-pac-pis1}
Manuel del Pino, Monica Musso, Frank Pacard, and Angela Pistoia.
\newblock Large energy entire solutions for the {Y}amabe equation.
\newblock {\em J. Differential Equations}, 251(9):2568--2597, 2011.


\bibitem{del-mus-pac-pis2}
Manuel del Pino, Monica Musso, Frank Pacard, and Angela Pistoia.
\newblock Torus action on {$S^n$} and sign-changing solutions for conformally
invariant equations.
\newblock {\em Ann. Sc. Norm. Super. Pisa Cl. Sci. (5)}, 12(1):209--237, 2013.

\bibitem{din}
Wei~Yue Ding.
\newblock On a conformally invariant elliptic equation on {${\bf R}^n$}.
\newblock {\em Comm. Math. Phys.}, 107(2):331--335, 1986.

\bibitem{esc1}
Jos\'{e}~F. Escobar.
\newblock Sharp constant in a {S}obolev trace inequality.
\newblock {\em Indiana Univ. Math. J.}, 37(3):687--698, 1988.

\bibitem{esc}
Jos\'{e}~F. Escobar.
\newblock Conformal deformation of a {R}iemannian metric to a scalar flat
metric with constant mean curvature on the boundary.
\newblock {\em Ann. of Math. (2)}, 136(1):1--50, 1992.

\bibitem{fel-ahm1}
Veronica Felli and Mohameden Ould~Ahmedou.
\newblock Compactness results in conformal deformations of {R}iemannian metrics
on manifolds with boundaries.
\newblock {\em Math. Z.}, 244(1):175--210, 2003.

\bibitem{fel-ahm2}
Veronica Felli and Mohameden Ould~Ahmedou.
\newblock A geometric equation with critical nonlinearity on the boundary.
\newblock {\em Pacific J. Math.}, 218(1):75--99, 2005.

\bibitem{fer-pet}
Juan~Carlos Fern\'{a}ndez and Jimmy Petean.
\newblock Low energy nodal solutions to the {Y}amabe equation.
\newblock {\em J. Differential Equations}, 268(11):6576--6597, 2020.

\bibitem{ghi-mic}
Marco~G. Ghimenti and Anna~Maria Micheletti.
\newblock A compactness result for scalar-flat metrics on low dimensional
manifolds with umbilic boundary.
\newblock {\em Calc. Var. Partial Differential Equations}, 60(3):Paper No. 119,
24, 2021.

\bibitem{gid-ni-nir}
B.~Gidas, Wei~Ming Ni, and L.~Nirenberg.
\newblock Symmetry and related properties via the maximum principle.
\newblock {\em Comm. Math. Phys.}, 68(3):209--243, 1979.

\bibitem{hen}
Guillermo Henry.
\newblock Second {Y}amabe constant on {R}iemannian products.
\newblock {\em J. Geom. Phys.}, 114:260--275, 2017.

\bibitem{khu-mar-sch}
M.~A. Khuri, F.~C. Marques, and R.~M. Schoen.
\newblock A compactness theorem for the {Y}amabe problem.
\newblock {\em J. Differential Geom.}, 81(1):143--196, 2009.

\bibitem{kim-mus-wei}
Seunghyeok Kim, Monica Musso, and Juncheng Wei.
\newblock Compactness of scalar-flat conformal metrics on low-dimensional
manifolds with constant mean curvature on boundary.
\newblock {\em Ann. Inst. H. Poincar\'{e} Anal. Non Lin\'{e}aire},
38(6):1763--1793, 2021.

\bibitem{lio}
P.-L. Lions.
\newblock The concentration-compactness principle in the calculus of
variations. {T}he limit case. {II}.
\newblock {\em Rev. Mat. Iberoamericana}, 1(2):45--121, 1985.
	
\bibitem{mar1}
Fernando~C. Marques.
\newblock Existence results for the {Y}amabe problem on manifolds with
boundary.
\newblock {\em Indiana Univ. Math. J.}, 54(6):1599--1620, 2005.

\bibitem{mar2}
Fernando~C. Marques.
\newblock Conformal deformations to scalar-flat metrics with constant mean
curvature on the boundary.
\newblock {\em Comm. Anal. Geom.}, 15(2):381--405, 2007.

\bibitem{may-ndi}
Martin Mayer and Cheikh~Birahim Ndiaye.
\newblock Barycenter technique and the {R}iemann mapping problem of
{C}herrier-{E}scobar.
\newblock {\em J. Differential Geom.}, 107(3):519--560, 2017.


\bibitem{mus-wei}
Monica Musso and Juncheng Wei.
\newblock Nondegeneracy of nodal solutions to the critical {Y}amabe problem.
\newblock {\em Comm. Math. Phys.}, 340(3):1049--1107, 2015.


\bibitem{pet}
Jimmy Petean.
\newblock On nodal solutions of the {Y}amabe equation on products.
\newblock {\em J. Geom. Phys.}, 59(10):1395--1401, 2009.

\bibitem{sch}
Richard Schoen.
\newblock Conformal deformation of a {R}iemannian metric to constant scalar
curvature.
\newblock {\em J. Differential Geom.}, 20(2):479--495, 1984.


\bibitem{ser}
James Serrin.
\newblock A symmetry problem in potential theory.
\newblock {\em Arch. Rational Mech. Anal.}, 43:304--318, 1971.

\bibitem{tru}
Neil~S. Trudinger.
\newblock Remarks concerning the conformal deformation of {R}iemannian
structures on compact manifolds.
\newblock {\em Ann. Scuola Norm. Sup. Pisa Cl. Sci. (3)}, 22:265--274, 1968.

\bibitem{wet}
Tobias Weth.
\newblock Energy bounds for entire nodal solutions of autonomous superlinear
equations.
\newblock {\em Calc. Var. Partial Differential Equations}, 27(4):421--437,
2006.

\bibitem{yam}
Hidehiko Yamabe.
\newblock On a deformation of {R}iemannian structures on compact manifolds.
\newblock {\em Osaka Math. J.}, 12:21--37, 1960.
	

	
\end{thebibliography}
\end{document}